\documentclass{amsart}
\usepackage{amssymb,amsfonts,amscd,amstext}
\usepackage[mathscr]{eucal}
\usepackage{epsfig,graphicx}
 
\newtheorem{theorem}{Theorem}[section]
\newtheorem{lemma}[theorem]{Lemma}
\newtheorem{proposition}[theorem]{Proposition}

\newtheorem*{theoremA}{Theorem A}
\newtheorem*{corollaryA}{Corollary A}
\newtheorem*{theoremB}{Theorem B}
\newtheorem*{corollaryB}{Corollary B}
\newtheorem*{exampleA}{Example A}
\newtheorem*{exampleA1}{Example A(i)}
\newtheorem*{exampleA2}{Example A(ii)}
\newtheorem*{exampleA3}{Example A(iii)}
\newtheorem*{exampleA4}{Example A(iv)}
\newtheorem*{exampleB}{Example B}
\newtheorem*{exampleC}{Example C}
\newtheorem*{exampleD}{Example D}

\theoremstyle{definition}

\numberwithin{equation}{section}

\newcommand{\C}{{\mathbb{C}}} 
\newcommand{\R}{{\mathbb{R}}}

\def\proof {{\it Proof.\,\,}}
\def\endproof {\rule[-0.5mm]{1.2ex}{1.2ex}}

\begin{document}

\title[On global linearization of  involutions]{On global linearization of planar involutions} 

\author{Benito Pires$^{\bigstar}$}
\address{Departamento de F\'isica e Matem\'atica, Universidade de S\~ao Paulo, 14040-901, Ribeir\~ao Preto, SP,Brazil}
\email{benito@ffclrp.usp.br}
\thanks{The first author was partially supported by FAPESP-BRAZIL (2009/02380-0
and  2008/02841-4).}

\author{Marco Antonio Teixeira}
\address{Departamento de Matem\'atica, IMECC/C.P. 6065, Universidade de Campinas - UNICAMP,
13083-970, Campinas, SP, Brazil}
\email{teixeira@ime.unicamp.br}
\thanks{The second author was partially supported by FAPESP-BRAZIL (2007/06896-5).}

\thanks{$^\bigstar$ Corresponding author (benito@ffclrp.usp.br)}

\date{May 20, 2011}

\begin{abstract} Let $\varphi:\R^2\to\R^2$ be an orientation--preserving $C^1$ involution such that $\varphi(0)=0$ and let ${\rm Spc}\,(\varphi)=\{{\rm Eigenvalues\,\,of}\,\, D\varphi(p)\mid p\in\R^2\}$. We prove that if ${\rm Spc}\,{(\varphi)}\subset\R$ or ${\rm Spc}\,(\varphi)\cap [1,1+\epsilon)=\emptyset$ for some $\epsilon>0$ then $\varphi$
is globally $C^1$ conjugate to the linear involution $D\varphi(0)$ via the conjugacy \mbox{$h=(I+D\varphi(0)\varphi)/2$}, where $I:\R^2\to\R^2$ is the identity map. Similarly, if $\varphi$ is an orientation-reversing $C^1$ involution such that
$\varphi(0)=0$ and ${\rm Trace}\,\big(D\varphi(0)D\varphi(p)\big)>-1 $ for all $p\in\R^2$ then $\varphi$ is globally $C^1$ conjugate to the linear involution $D\varphi(0)$ via the conjugacy $h$. Finally, we show that $h$ may fail to be a global linearization of $\varphi$ if the above conditions are not fulfilled.
\end{abstract}


\maketitle

\section{Introduction}

Let $\varphi:\R^2\to\R^2$ be a $C^1$ involution, that is, a $C^1$ map such that $\varphi\circ\varphi=I$, where
$I:\R^2\to\R^2$ is the identity map. It is widely known (see \cite{Mac,MZ}) that if $p$ belongs to ${\rm Fix}\,(\varphi)$, the fixed point set of $\varphi$, then around $p$ the involution $\varphi$ is locally \mbox{$C^1$ conjugate} to its linear part $D\varphi(p)$ via the conjugacy \mbox{$h=(I+D\varphi(p)\varphi)/2$}. 
In other words, there exist neighborhoods $U_p$ of $p$ and  $V_{h(p)}$ of $h(p)\in{\rm Fix}\,(D\varphi(p))$ such that $h\vert_{U_p}:U_p\to V_{h(p)}$ is an orientation-preserving $C^1$ diffeomorphism satisfying $h\circ\varphi(q)=D\varphi(p)h(q)$
for all $q\in U_p$. In this article we provide conditions on the involution $\varphi$ under which $h$ is a global linearization of $\varphi$, that is, $U_p=\R^2$. The key point is to find conditions which ensure the global injectivity of $h$. In general, $h$ may fail to be (globally) injective. 

The results of this article can also be stated in terms of the existence of a \mbox{$\varphi$-invariant} foliation for a planar involution $\varphi$. If $\varphi(0)=0$ and
$\varphi\neq I$ then there exists a $\varphi$-invariant local foliation of a neighborhood $U$ of $0$ topologically equivalent to the radial foliation (if $\varphi$ is orientation--preserving) or
to the vertical foliation (if $\varphi$ is orientation--reversing). In both cases there exists an explicit formula for the \mbox{$\varphi$-invariant} local foliation in terms of $\varphi$. We present sufficient conditions for the existence of a global foliation of $\R^2$ that is invariant by $\varphi$ and coincides with the local foliation around $0$.
 
It is not always possible to extend the local foliation induced by $h$ to a  $\varphi$-invariant global foliation. In order to do so, it is necessary to control the behaviour of the involution away from its fixed point set. More specifically, we need to know how the set ${\rm Spc}\,(\varphi)$ lies in the complex plane $\mathbb{C}$.
All it is  known is that nearby ${\rm Fix}\,({\varphi})$ the eigenvalues of $D\varphi(p)$ are close to $\{-1,1\}$. Away from ${\rm Fix}\,{(\varphi)}$, the involution $\varphi$ looks like a general nonsingular planar map.
We will need to impose some restrictions on ${\rm Spc}\,(\varphi)$ (the spectral conditions)
to get the stated results.

To reach the results we apply the theory of injectivity of planar maps to the problem of global linearization of $C^1$ involutions. Concerning injectivity results for planar maps, we would like to mention the outstanding work of Fessler \cite{F} and \mbox{Gutierrez \cite{G1}}, who solved affirmatively the bidimensional Markus--Yamabe Conjecture in 1995. Thereafter, many improvements (see \cite{CGL,FGR1,FGR2, GJLT}), variations (see \cite{BS}) and applications (see  \cite{AGG, AGA,GR, GPR,GS,R}) of that result have appeared within the mathematical literature. Other important approaches to the topic injectivity may be found in \cite{CM, C, GT, MY, O,  SX}.

With respect to linearization of involutions, there are two worth mentioning results: the Bochner-Montgomery Theorem (see \cite{MZ}) about the linearization of a compact group of transformations around a fixed point and the result of \mbox{K. Meyer \cite{KMey}} about the linearization of an antisymplectic involution of a symplectic manifold in a neighborhood of its fixed point submanifold. In addition to these, we refer the reader to the articles \cite{MMT} and \cite{T}, which deal with simultaneous local linearization of pairs of involutions.

\section{Statement of the results}

We say that a $C^1$ map $f:\R^2\to\R^2$ is {\it orientation--preserving} (respectively {\it orientation--reversing}) as
${\rm Det}\,(Df)>0$ (respectively ${\rm Det}\,(Df)<0$). We will make use of the following sets:
the fixed point set of $f$,
$${\rm Fix}\,(f)=\{p\in\R^2\mid f(p)=p\},$$
and the Spectrum of $f$,
$${\rm Spc}\,(f)=\{{\rm Eigenvalues\,\,of}\,\,Df(p)\mid p\in\R^2\}.$$

We say that a $C^r$ involution $\varphi:\R^2\to\R^2$ is {\it (globally) $C^r$ linearizable} if there exists an injective, orientation--preserving $C^r$ map $h:\R^2\to\R^2$ and $p\in{\rm Fix}\,(\varphi)$
such that $h\circ \varphi=D\varphi(p) h$.  

Throughout this paper, we assume that $0\in{\rm Fix}\,(\varphi)$. There is no loss of generality in doing that as all planar $C^1$ involution has a fixed point (see Proposition \ref{notempty}).

Given a $C^1$ involution $\varphi:\R^2\to\R^2$ such that $\varphi(0)=0$, let $h:\R^2\to\R^2$ be the $C^1$ map defined by:
$$
h=\frac{1}{2} \big( I+D\varphi(0)\varphi \big).
$$ 
We call $h$ the {\it standard map}. It turns out to be useful to write $h=\frac12 D\varphi(0)g$, where
$g=D\varphi(0)+\varphi$. It is easy to check that $h\circ \varphi=D\varphi(0)h$. Moreover, since $D\varphi(0)=I$, by the Inverse Function Theorem $h$ is injective in a neighborhood of $0$ and so it is a local linearization of $\varphi$. 

Since all orientation-preserving (respectively orientation-reversing) linear involution $\varphi\neq I$ is linearly conjugate to $-I$ (respectively to $(x,y)\mapsto (x,-y)$), there exists a unique foliation ${\mathscr F}_{D\varphi(0)}$ of $\R^2\setminus\{0\}$ by rays (respectively of $\R^2$ by lines) invariant by the \mbox{involution $D\varphi(0)$}.
We let $\mathscr{F}_{\varphi}$ denote the family of sets which are the inverse images of the leaves of $\mathscr{F}_{D\varphi(0)}$
by $h$. It is plain that if $h$ is a local $C^1$ diffeomorphism then $\mathscr{F}_\varphi$ is a $\varphi$-invariant $C^1$ foliation.

Now we state our results.
 
\begin{theoremA}\label{tmain1} An orientation--preserving $C^1$ involution $\varphi:\R^2\to\R^2$ such that $\varphi(0)=0$ is globally \mbox{$C^1$ lineari\-zable} via the standard map under any of the following hypotheses:
\begin{itemize}
\item [(a)] ${\rm Spc}\,(\varphi)=\{1\}$ $($in this case $\varphi=I)$ or
\item [(b)] ${\rm Spc}\,(\varphi)\cap [1,1+\epsilon)=\emptyset$ for some $\epsilon>0$ or
\item [(c)] ${\rm Spc}\,(\varphi)\subset\R$
\end{itemize}
\end{theoremA}

\begin{theoremB}\label{tmain3} An orientation--reversing $C^1$ involution $\varphi:\R^2\to\R^2$ such that $\varphi(0)=0$ is globally $C^1$ linearizable via the standard map if \mbox{${\rm Trace}\,\big(D\varphi(0)D\varphi(p)\big)>-1$} for all $p\in\R^2$. 
\end{theoremB}

\begin{corollaryA} Let $\varphi:\R^2\to\R^2$ be an orientation--preserving $C^1$ involution such that $\varphi(0)=0$ and $\varphi\neq I$. If either ${\rm Spc}\,(\varphi)\cap [1,1+\epsilon)=\emptyset$ for some $\epsilon>0$ or ${\rm Spc}\,(\varphi)\subset\R$
then $\mathscr{F}_\varphi$ is a $\varphi$-invariant $C^1$ foliation of $\R^2\setminus\{0\}$ topologically equivalent to the radial foliation.
\end{corollaryA}

\begin{corollaryB} Let $\varphi:\R^2\to\R^2$ be an orientation--reversing $C^1$ involution such that $\varphi(0)=0$. 
If ${\rm Trace}\,\big(D\varphi(0)D\varphi(p)\big)>-1 $ for all $p\in\R^2$ then $\mathscr{F}_\varphi$
is a $\varphi$-invariant \mbox{$C^1$ foliation} of $\R^2$ topologically equivalent to the vertical foliation.
\end{corollaryB}

We also include a series of examples showing that the results of this paper may be applied
to many involutions. Besides, $C^1$ global linearization of a $C^1$ involution via the standard map may fail if the spectral conditions are not fulfilled. These examples are examined in detail in the final section.

\begin{exampleA} The polynomial involutions $\varphi,\psi,\phi,\xi:\R^2\to\R^2$ defined below $($for all integer $n\ge 0)$ are globally $C^\infty$ linearizable via the standard map.
\begin{itemize}
\item [(i)] $\varphi(x,y)=(x-y^{2n+1},-y);$
\item [(ii)] $\psi(x,y)=(-x+y^{2n},-y);$
\item [(iii)] $\phi(x,y)= \Big(-y-\Big(\dfrac{x+y}{2}\Big)^{2n+1},-x+\Big(\dfrac{x+y}{2}\Big)^{2n+1}\Big);$
\item [(iv)] $\xi(x,y)= \Big(-x+\Big(\dfrac{x+y}{2}\Big)^{2n},-y-\Big(\dfrac{x+y}{2}\Big)^{2n}\Big).$
\end{itemize}
\end{exampleA}

\vspace{0.2cm}

Notice that in the Example A, $\varphi$ and $\phi$ are orientation--reversing involutions whereas
$\psi$ and $\xi$ are orientation--preserving ones. Thus, $\mathscr{F}_{\varphi}$ and $\mathscr{F}_\phi$
are foliations of $\R^2$ by topological lines whereas $\mathscr{F}_\psi$ and $\mathscr{F}_\xi$ are foliations of $\R^2\setminus\{0\}$ by rays.
 \vspace{0.2cm}

\begin{exampleB} The  orientation--reversing $C^\infty$ involution $\varphi:\R^2\to\R^2$ defined by
$$\varphi(x,y)=\left (
{\rm arcsinh}\,\Big (\dfrac{ {\rm sinh}\,(x)+{\rm sinh}\,(y)}{2} \Big),
{\rm arcsinh}\, \Big (\dfrac{3{\rm sinh}\,(x)-{\rm sinh}\,(y)}{2} \Big ) \right )
$$
is globally $C^\infty$ linearizable via the standard map and $\mathscr{F}_
\varphi$ is topologically equivalent to the vertical foliation.
\end{exampleB}

\begin{exampleC} There exists an orientation-preserving $C^\infty$ involution $\varphi:\R^2\to\R^2$ that is not globally $C^1$
linearizable via the standard map. In this case, $\mathscr{F}_\varphi$ is not topologically
equivalent to the radial foliation.
\end{exampleC}

\begin{exampleD} There exists an orientation-reversing $C^\infty$ involution $\varphi:\R^2\to\R^2$ that is not globally $C^1$ linearizable
via the standard map. In this case, $\mathscr{F}_\varphi$ is not topologically equivalent to the vertical foliation.
\end{exampleD}

\section{Foliations invariant by involutions}

Let $M$ be either the whole plane $\R^2$ or the punctured plane $\R^2\setminus\{0\}$.
Let $\mathcal{F}=\{L_\alpha\}_{\alpha\in A}$ be a partition of $M$ into disjoint connected subsets called {\it leaves}. We say that $\mathcal{F}$ is a {\it $C^1$ foliation of M by curves} if every point $p\in M$ has
a neighborhood $U$ and a $C^1$ diffeomorphism $(f_1,f_2):U\to U_1\times U_2\subset\R^2$
such that for each leaf $L_\alpha$, the connected components of $L_\alpha\cap U$ are level curves of $f_1$. 

We say that a foliation $\mathcal{F}$ is {\it invariant by} an involution $\varphi:\R^2\to\R^2$, or that $\mathcal{F}$ is {\it $\varphi$-invariant} if $\varphi$ takes the leaves of $\mathcal{F}$ onto the leaves of $\mathcal{F}$.

Let $\varphi:\R^2\to\R^2$ be a $C^1$ involution. Let $\gamma:(0,+\infty)\to\R^2$ be an embedding of $(0,+\infty)$. As usual, 
we identify $\gamma$ with $\gamma ((0,+\infty))=\{\gamma(t)\mid t\in (0,+\infty)\}$. We say that $\gamma$ is a ray if $\gamma(0)=0$ and $\lim_{t\to\infty}\Vert \gamma(t)\Vert=+\infty$.

We say that a foliation $\mathcal{F}$ of $\R^2\setminus\{0\}$ (respectively of $\R^2$) is {\it topologically equivalent to the radial foliation} (respectively {\it topologically equivalent to the vertical foliation}) if there exists a homeomorphism
$h:\R^2\to h(\R^2)$ which takes each leaf of $\mathcal{F}$ in a ray (respectively in a vertical line). In the first case, $h(0)=0$.

\begin{proposition} Let $L:\R^2\to\R^2$ be a linear involution. The following statements are true:
\begin{itemize}
\item [(a)] If $L\neq I$ is orientation-preserving then there exists a $L$-invariant foliation of $\R^2\setminus\{0\}$ topologically equivalent to the radial foliation;
\item [(b)] If $L$ is orientation-reversing then there exists a $L$-invariant foliation of $\R^2$ topologically equivalent to the vertical foliation.
\end{itemize}
\end{proposition}
\begin{proof} Every linear involution is diagonalizable and so is conjugate to one of the involutions: I, -I or
$(x,y)\mapsto (x,-y)$. 
\end{proof}

\section{Injectivity results}

Sufficient conditions for global injectivity of $C^1$ maps were provided independently by Fessler \cite{F} and Gutierrez \cite{G1} who proved that if $f:\R^2\to\R^2$ is a \mbox{$C^1$ map} such that
${\rm Spc}\,(f)\cap [0,\infty)=\emptyset$ then $f$ is injective. Later on, \mbox{Cobo--Gutierrez--Llibre \cite{CGL}}
obtained the same result under the condition ${\rm Spc}\,(f)\cap (-\epsilon,\epsilon)=\emptyset$ for some $\epsilon>0$. A slight  variation of this result is the following theorem: 

\begin{theorem}\label{fundamental} Let $f:\R^2\to\R^2$ be a $C^1$ map. If for some $\epsilon>0$,  ${\rm Spc}\,(f)\cap
[0,\epsilon)=\emptyset$ or  ${\rm Spc}\,(f)\cap
(-\epsilon,0]=\emptyset$ then $f$ is injective.
\end{theorem}

Notice that Theorem \ref{fundamental} still holds true even if $f$ is only differentiable (see \cite{FGR2}). 


\section{Fixed point set of $C^1$ involutions}

The first step towards understanding the global behaviour of an involution  is to clarify the structure of its fixed point set. Given two points $p,q\in\R^2$, we let:
$$[p,q]=\{(t-1)p+tq\mid t\in [0,1]\}$$
denote the line segment joining $p$ and $q$. 

\begin{proposition}\label{notempty} If $\varphi:\R^2\to\R^2$ is a $C^1$ involution then ${\rm Fix}\,(\varphi)\neq\emptyset$.
\end{proposition}
\begin{proof} Suppose that ${\rm Fix}\,(\varphi)=\emptyset$. We claim that there exist $p_1\in\R^2$ and a line segment \mbox{$\sigma=[p_1,\varphi(p_1)]$} joining $p_1$ and $\varphi(p_1)$ such that $\sigma\cup\varphi(\sigma)$ is a $\varphi$-invariant topological circle.
Indeed, let $p_0\in\R^2$ and let $\sigma_0=[p_0,\varphi(p_0)]$ be the line segment joining $p_0$ and
$\varphi(p_0)$. Let $S=\{p\in\sigma_0\mid \varphi(p)\in\sigma_0\}$. We have that $S$ is a compact non--empty set.
Besides, \mbox{$r=\inf_{p\in S} \vert p-\varphi(p)\vert>0$}, otherwise $\varphi$ would have a fixed point.
Now let $p_1\in S$ be such that $\vert p_1-\varphi(p_1)\vert=r$. It is plain that $[p_1,\varphi(p_1)]\cap \varphi([p_1,\varphi(p_1)])=\{p_1,\varphi(p_1)\}$. Hence, if $\sigma_1=[p_1,\sigma(p_1)]$ then $\sigma_1\cup \varphi(\sigma_1)$ is a $\varphi-$invariant topological circle. This proves the claim. Let $K$ be the compact region bounded by the topological circle $\sigma_1\cup\varphi(\sigma_1)$. Because $\varphi$ is a diffeomorphism of $\R^2$ taking the boundary of $K$ into itself, we have that $\varphi(K)=K$. By the Jordan Curve Theorem, we have that $K$ is homeomorphic to the closed unit ball. By the Brouwer Fixed Point Theorem, $\varphi$ has  a fixed point in $K$, which is a contradiction.
\end{proof}\\

Proposition \ref{notempty} in the case of orientation--preserving involutions is a consequence of a
stronger result (see \cite[Wandering Theorem, p. 102 ]{H}).




Let $\varphi:\R^2\to\R^2$ be an orientation--preserving $C^1$ involution. We set:
\begin{eqnarray*}
{\rm Fix}^+\,(\varphi)&=&\{p\in{\rm Fix}\,(\varphi)\mid D\varphi(p)=I \},\\
{\rm Fix}^-\,(\varphi)&=&\{p\in{\rm Fix}\,(\varphi)\mid D\varphi(p)=-I \}.
\end{eqnarray*}
\begin{lemma}\label{simples} If $\varphi:\R^2\to\R^2$ is an orientation--preserving $C^1$ involution
then ${\rm Fix}\,(\varphi)={\rm Fix}^+\,(\varphi)\cup{\rm Fix}^-\,(\varphi).$
\end{lemma}
\begin{proof} It follows from $\varphi$ being an orientation-preserving $C^1$ involution that
if $p\in{\rm Fix}\,({\varphi})$ then either ${\rm Spc}\,(D\varphi(p))=\{1\}$
or ${\rm Spc}\,(D\varphi(p))=\{-1\}$. Hence, by the Jordan Form Theorem, there exist $a\in\{-1,1\}$ and $b\in \{0,1\}$ such that
$[D\varphi(p)]_{\mathcal B}=\left (
\begin{matrix}
a & b\\ 0 & a
\end{matrix}\right ),
$
where $[D\varphi(p)]_{\mathcal{B}}$ denote the matrix of $D\varphi(p)$ with respect to the basis of eigenvectors $\mathcal{B}$.
As $D\varphi(p)$ is an involution, we have that $b=0$, and so $D\varphi(p)=aI$.
\end{proof}\\

\begin{lemma}\label{posde} 
If $\varphi:\R^2\to\R^2$ is an orientation-preserving $C^1$ involution then either ${\rm Fix}\,(\varphi)=
{\rm Fix}^+(\varphi)=\R^2$ or ${\rm Fix}\,(\varphi)={\rm Fix}^-(\varphi)\varsubsetneq\R^2$.
\end{lemma}
\begin{proof} By Proposition \ref{notempty}, ${\rm Fix}\,(\varphi)\neq\emptyset$. Being both open (by local linearization) and closed, ${\rm Fix}^+\,(\varphi)$ is either empty or the whole $\R^2$. In this way, by
Lemma \ref{simples}, either ${\rm Fix}\,(\varphi)=
{\rm Fix}^+(\varphi)=\R^2$ or ${\rm Fix}\,(\varphi)={\rm Fix}^-(\varphi)$.
Finally, by local linearization, if $p\in {\rm Fix}^-(\varphi)$ then $p$ is an isolated fixed point of $\varphi$.
Thus ${\rm Fix}^-(\varphi)\varsubsetneq\R^2$.
\end{proof}

\begin{proposition}\label{fpop} If $\varphi:\R^2\to\R^2$ is an orientation--preversing $C^1$ involution then
either ${\rm Fix}\,\,(\varphi)=\R^2$ or ${\rm Fix}\,(\varphi)$ is a unitary set.
\end{proposition}
\begin{proof} Suppose that ${\rm Fix}(\varphi)\neq\R^2$. By Lemma \ref{posde}, ${\rm Fix}\,(\varphi)=
{\rm Fix}^-(\varphi)$. Consequently, by local linearization, ${\rm Fix}\,(\varphi)$ is a discrete set (formed by isolated fixed points). We claim that ${\rm Fix}\,(\varphi)$ is a unitary set. By Proposition \ref{notempty}, ${\rm Fix}\,(\varphi)\neq\emptyset$. Without loss of generality we may assume that $\varphi(0)=0$. Let $\gamma\subset\R^2$ be a ray such that $\gamma\setminus\{0\}\subset\R^2\setminus {\rm Fix}\,(\varphi)$. Firstly let us consider the case in which $\gamma\cap\varphi(\gamma)=\{0\}$. In this case, $\gamma\cup\varphi(\gamma)$ is a $\varphi$-invariant topological line separating $\R^2$. Because $D\varphi(0)=-I$, we have by local linearization that $\varphi$ takes one connected component of $\R^2\setminus (\gamma\cup\varphi(\gamma))$ into the other one. Consequently, ${\rm Fix}\,(\varphi)=\{0\}$. Therefore, we may assume that
$\gamma\cap\varphi(\gamma)\supsetneq \{0\}$. Let $t_0=\inf\,\{t\in (0,\infty)\mid \gamma(0,t]\cap\varphi(\gamma(0,t])\neq\emptyset\}$. Because $\varphi$ is locally conjugate to $-I$ around $0$, we have
that $t_0>0$. Besides, by the choice of $t_0$, $C=\gamma \big([0,t_0]\big)\cup \varphi\big(\gamma\big ([0,t_0]\big) \big)$ is a $\varphi$-invariant topological circle passing through $0$. We affirm that this is impossible. Indeed, let $K$ be the compact set bounded by the topological circle $C$. By the Jordan Curve Theorem, $K$ is $\varphi$-invariant. However, this is a contradiction because $0\in C$ and around $0$ the involution $\varphi$ is topologically conjugate to the linear \mbox{involution $-I$.} 
\end{proof}\\

In this paper, Proposition \ref{fpop} will follow automatically from other hypotheses.
For instance, the condition ${\rm Spc}\,(f)\cap [1,1+\epsilon)=\emptyset$ implies that
$\#{\rm Fix}\,(f)\le 1$ ($\#$ denotes the cardinality) where $f:\R^2\to\R^2$ is any \mbox{$C^1$ map} (see \cite[Corollary 2, p. 421]{AGA}).

\section{Proof of the main results}

Henceforth, we will assume that $0$ is a fixed point of the involution $\varphi:\R^2\to\R^2$ (see Proposition \ref{notempty}). We will need some lemmas and propositions for the proof of the main results.  

\begin{lemma}\label{sr} Let $\varphi:\R^2\to\R^2$ be a $C^1$ involution. There exist continuous functions
$\lambda_j:\R^2\to\C$, $j\in\{1,2\}$, such that ${\rm Spc}\,(D\varphi(p))=\{\lambda_1(p),\lambda_2(p)\}$
for all $p\in\R^2$.
\end{lemma}
\proof It follows from the Theory of Ordinary Differential Equations that the functions
$\lambda_j$, $j\in\{1,2\}$, are given by:
\begin{equation*}
\lambda_j(p)=\dfrac{{\rm Trace}\,(D\varphi(p))+(-1)^{j}\sqrt{   \big({\rm Trace}\,(D\varphi(p))  
\big)^2-4{\rm Det}\,(D\varphi(p))                    }          }{2}.
\end{equation*} \endproof

\begin{proposition}\label{pmain1} Let $\varphi:\R^2\to\R^2$ be an orientation--preserving $C^1$ involution such that $\varphi(0)=0$.
If ${\rm Spc}\,(\varphi)\cap [1,1+\epsilon)=\emptyset$ for some $\epsilon>0$  then $\varphi$ is globally $C^1$ linearizable via the standard map.
\end{proposition}

\begin{proof}  By Lemma \ref{posde} and Proposition \ref{fpop}, either ${\rm Fix}\,(\varphi)={\rm Fix}^+\,(\varphi)=\R^2$ (and so $\varphi=I$) or ${\rm Fix}\,(\varphi)={\rm Fix}^-\,(\varphi)$ is a unitary set. In this case, ${\rm Fix}\,(\varphi)=\{0\}$ and $D\varphi(0)=-I$.
 By Lemma \ref{sr}, there exist continuous functions $\lambda_j:\R^2\to\C$, $j\in\{1,2\}$, such that
 ${\rm Spc}\,(D\varphi(p))=\{\lambda_1(p),\lambda_2(p)\}$ for all $p\in\R^2$.
 By the Jordan Form Theorem, for each $p\in\R^2$, there exist a linear isomorphism $S(p):\C^2\to\C^2$ and a diagonal linear operator
$J(p):\C^2\to\C^2$ defined by $J(p)(u,v)=(\lambda_1(p) u,\lambda_2(p) v)$ such that
\begin{eqnarray}\label{thiseqq}
\hspace{1cm} D\varphi(0)+D\varphi(p)=-I+S(p) J(p) (S(p))^{-1}=S(p) (-I+J(p)) (S(p))^{-1}. 
\end{eqnarray}
It follows at once from (\ref{thiseqq}) that ${\rm Spc}\,(g)={\rm Spc}\,(\varphi)-1$, where 
$g:\R^2\to\R^2$ is the map $g=D\varphi(0)+\varphi$. The hypothesis ${\rm Spc}\,(\varphi)\cap [1,1+\epsilon)=\emptyset$ implies that ${\rm Spc}\,(g)\cap [0,\epsilon)=\emptyset$. It follows from Theorem \ref{fundamental}  that $g$ is globally injective. Thus $h=\frac12D\varphi(0)g$ is also globally injective.
\end{proof} \\
\begin{proposition}\label{pmain2} Let $\varphi:\R^2\to\R^2$ be an orientation--preserving $C^1$ involution such that $\varphi(0)=0$.
If ${\rm Spc}\,(\varphi)\subset\R$ then $\varphi$ is globally $C^1$ linearizable via the standard map.
\end{proposition}
\begin{proof} The same proof of Proposition \ref{pmain1} works with some slight changes. We have that ${\rm Fix}\,(\varphi)=\{0\}$ and $D\varphi(0)=-I$.
 By Lemma \ref{sr}, there exist continuous functions $\lambda_j:\R^2\to\R\setminus\{0\}$, $j\in\{1,2\}$, such that
 ${\rm Spc}\,(D\varphi(p))=\{\lambda_1(p),\lambda_2(p)\}\subset\R$ for all $p\in\R^2$.
 By the Jordan Form Theorem, for each $p\in\R^2$, there exist a linear isomorphism $S(p):\R^2\to\R^2$ and a lower triangular linear operator
$J(p):\R^2\to\R^2$ defined by $J(p)(u,v)=(\lambda_1(p) u+c(p)v,\lambda_2(p) v)$, for some $c(p)\in \{0,1\}$, such that
\begin{eqnarray}\label{thiseq}
\hspace{1cm} D\varphi(0)+D\varphi(p)=-I+S(p) J(p) (S(p))^{-1}=S(p) (-I+J(p)) (S(p))^{-1}. 
\end{eqnarray}
It follows at once from (\ref{thiseq}) that ${\rm Spc}\,(g)={\rm Spc}\,(\varphi)-1$, where 
$g:\R^2\to\R^2$ is the map $g=D\varphi(0)+\varphi$. Because $\lambda_1(0)=-1=\lambda_2(0)$
and ${\rm Spc}\,(\varphi)\subset\R\setminus\{0\}$, we have that $\lambda_1(p)<0$ and $\lambda_2(p)<0$ for all
$p\in\R^2$. Hence, by the above, ${\rm Spc}\,(g)\subset (-\infty,-1]$. Theorem \ref{fundamental} implies that $g$ is globally injective. Hence, $h=\frac12D\varphi(0)g$ is also globally injective.
\end{proof}

\begin{theoremA}\label{tmain1} An orientation--preserving $C^1$ involution $\varphi:\R^2\to\R^2$ such that $\varphi(0)=0$ is globally \mbox{$C^1$ lineari\-zable} via the standard map under any of the following hypotheses:
\begin{itemize}
\item [(a)] ${\rm Spc}\,(\varphi)=\{1\}$ $($in this case $\varphi=I)$ or
\item [(b)] ${\rm Spc}\,(\varphi)\cap [1,1+\epsilon)=\emptyset$ for some $\epsilon>0$ or
\item [(c)] ${\rm Spc}\,(\varphi)\subset\R$
\end{itemize}
\end{theoremA}
\begin{proof} It follows immediately from Propositions \ref{pmain1} and \ref{pmain2}.
\end{proof}

\begin{corollaryA} Let $\varphi:\R^2\to\R^2$ be an orientation--preserving $C^1$ involution such that $\varphi(0)=0$ and $\varphi\neq I$. If either ${\rm Spc}\,(\varphi)\cap [1,1+\epsilon)=\emptyset$ for some $\epsilon>0$ or ${\rm Spc}\,(\varphi)\subset\R$
then $\mathscr{F}_\varphi$ is a $\varphi$-invariant $C^1$ foliation of $\R^2\setminus\{0\}$ topologically
equivalent to the radial foliation. 
\end{corollaryA}
\begin{proof}  By the proof of Proposition \ref{pmain1}, $D\varphi(0)=-I$. By \mbox{Theorem A}, the standard map 
\mbox{$h=\frac12 (I+D\varphi(0)\varphi$)} is a \mbox{$C^1$ conjugacy} between $\varphi$ and $-I$. By definition,
$\mathscr{F}_\varphi$ is the pullback by $h$ of the radial foliation $\mathscr{F}_{D\varphi(0)}$. In this way, because $h$ is a local $C^1$ diffeomorphism, $\mathscr{F}_{\varphi}$ is a $C^1$ foliation of $\R^2\setminus\{0\}$ by rays. 
\end{proof}

\begin{theoremB}\label{tmain3} An orientation--reversing $C^1$ involution $\varphi:\R^2\to\R^2$ such that $\varphi(0)=0$ is globally  $C^1$ linearizable via the standard map if  \mbox{${\rm Trace}\,\big(D\varphi(0)D\varphi(p)\big)>-1$} for all $p\in\R^2$. 
\end{theoremB}
\begin{proof} Let $\mathcal{B}$ be the eigenvectors basis with respect to which the representation matrix of $D\varphi(0)$ is the diagonal matrix:
\begin{equation*}\label{ddd}
[D\varphi(0)]_{{\mathcal{B}}}=\begin{pmatrix}
1 & 0\\
0 & -1
\end{pmatrix}.
\end{equation*}
Notice that ${\rm Trace}\,(Dh(p))=\frac12\Big(2+{\rm Trace}\,\big(D\varphi(0)D\varphi(p)\big)\Big)>\frac12>0$
for all $p\in\R^2$. Now fix $p\in\R^2$ and let $a,b,c,d\in\R$ be such that
\begin{equation*}\label{ddd}
[D\varphi(p)]_{{\mathcal{B}}}=\begin{pmatrix}
a & b\\
c & d
\end{pmatrix}.
\end{equation*}
Notice that $a-d={\rm Trace}\,([D\varphi(0)]_{\mathcal{B}}
[D\varphi(p)]_{\mathcal{B}})={\rm Trace}\,(D\varphi(0)D\varphi(p))>-1$. Hence,
$d-a-1<0$.
On the other hand
$${\rm Det}\,(D\varphi(0)+D\varphi(p))=(a+1)(d-1)-bc=(ad-bc)+d-a-1=\underbrace{{\rm Det}\,(D\varphi(p))}_{<0}+\underbrace{d-a-1}_{<0}<0.$$ Therefore, ${\rm Det}\,(Dh(p))=\frac14{\rm Det}\,\big(I+D\varphi(0)D\varphi(p)\big)=\frac14 {\rm Det}\,\big(D\varphi(0)\big)\cdot {\rm Det}\big(D\varphi(0)+D\varphi(p) \big)>0$ for all $p\in\R^2$. The inequalities ${\rm Trace}\,(Dh)>0$ and ${\rm Det}\,(Dh)>0$
 imply that \mbox{${\rm Spc}\,(h)\cap (-\infty,0]=\emptyset$}. By Theorem \ref{fundamental}, $h$ is globally injective.
\end{proof}

\begin{corollaryB} Let $\varphi:\R^2\to\R^2$ be an orientation--reversing $C^1$ involution such that $\varphi(0)=0$. 
If ${\rm Trace}\,\big(D\varphi(0)D\varphi(p)\big)>-1 $ for all $p\in\R^2$ then $\mathscr{F}_\varphi$
is a $\varphi$-invariant \mbox{$C^1$ foliation} of $\R^2$ topologically equivalent to the vertical foliation.
\end{corollaryB}
\begin{proof} Every linear involution is linearly diagonalizable. Hence, there exists a linear transformation
$S:\R^2\to\R^2$ such that $SD\varphi(0)S^{-1}=L$, where $L(x,y)=(x,-y)$. By Theorem B, $\varphi$ is globally $C^1$ linearizable via $h$. These two facts together imply that
$Sh\circ \varphi=LSh$. In other words, $\{p\in\R^2\mid \pi_1\big(Sh(p)\big)={\rm constant}\}$ is a $\varphi$-invariant global foliation
of $\R^2$ topologically equivalent to the vertical foliation $\mathcal{F}_{L}$, where $\pi_1:\R^2\to\R$ is the canonical projection $(x,y)\mapsto x$. It is not difficult to show that such a foliation coincides
with $\mathcal{F}_\varphi$.
\end{proof}

\section {Examples}

In this section, we prove that the examples presented in the section ``Statement of the results" have
the properties announced therein. In what follows, we let $p=(x,y)$. The canonical basis of $\R^2$ is denoted by $E$. 

\begin{exampleA1} Let $\varphi:\R^2\to\R^2$ be the orientation--reversing involution defined by
$\varphi(x,y)=(x-y^{2n+1},-y)$.
The matrix of $D\varphi(0)D\varphi(p)$ is given by:
\begin{equation*}
[D\varphi(0)D\varphi(p)]_E=
\begin{pmatrix}
1 &  -(2n+1) y^{2n} \\
0  &  1
\end{pmatrix}.
\end{equation*}
Thus ${\rm Trace}\,(D\varphi(0)D\varphi(p))=2>-1$. 
By Theorem B,  $h:\R^2\to\R^2$ defined by \mbox{$h(x,y)=\Big(x-\dfrac{y^{2n+1}}{2},y\Big)$} is a $C^\infty$ linearization \mbox{of $\varphi$}. By Corollary B, the family of curves 
$$\{2x-y^{2n+1}={\rm constant}\}$$ is a $\varphi$-invariant \mbox{$C^\infty$ foliation} of $\R^2$ topologically equivalent to the vertical foliation.
\end{exampleA1}

\begin{exampleA2} Let $\psi:\R^2\to\R^2$ be the orientation--preserving involution defined by
$\psi(x,y)=(-x+y^{2n},-y)$. The matrix  of $D\psi(p)$ is:
\begin{equation*}
[D\psi(p)]_E=\begin{pmatrix}
-1 & 2ny^{2n-1}\\
0 & -1
\end{pmatrix}.
\end{equation*}
Therefore, ${\rm Spc}(\psi)=\{-1\}\subset\R$ and $(c)$ of Theorem A implies that $\psi$ is globally
$C^1$ linearizable via the standard map. In this case, $h(x,y)=\Big(x-\dfrac{y^{2n}}{2},y\Big)$ is a global $C^1$ linearization of $\psi$ and, by Corollary A, $\mathscr{F}_\psi$ is topologically equivalent to the radial foliation.
\end{exampleA2}

\begin{exampleA3} Let $\phi:\R^2\to\R^2$ be the orientation--reversing involution defined by
$$\phi(x,y)= \Big(-y-\Big(\frac{x+y}{2}\Big)^{2n+1},-x+\Big(\frac{x+y}{2}\Big)^{2n+1}\Big).$$
The matrix of $D\phi(0)D\phi(p)$ is:
\begin{equation*}
[D\phi(0)D\phi(p)]_E=\begin{pmatrix}
1-\alpha(p) & -\alpha(p)\\
\alpha(p) &  1+\alpha(p)
\end{pmatrix},
\end{equation*}
where $\alpha(x,y)=\dfrac12(2n+1)\Big(\dfrac{x+y}{2}\Big)^{2n}$. Thus
${\rm Trace}\,(D\phi(0)D\phi(p))=2>-1$.
By Theorem B and by \mbox{Corollary B},  $\phi$ is globally $C^\infty$ linearizable via the standard map and that $\mathscr{F}_\phi$ is topologically equivalent to the vertical foliation.
\end{exampleA3}

\begin{exampleA4} Let $\xi:\R^2\to\R^2$ be the orientation-preserving involution defined by
$$\xi(x,y)= \Big(-x+\Big(\frac{x+y}{2}\Big)^{2n},-y-\Big(\frac{x+y}{2}\Big)^{2n}\Big).$$
The matrix of $D\xi(p)$ is given by:
\begin{equation*}
[D\xi(p)]_E=\begin{pmatrix}
\beta(p)-1 & \beta(p)\\
-\beta(p) &  -1-\beta(p)
\end{pmatrix},
\end{equation*}
where $\beta(x,y)=n\Big(\dfrac{x+y}{2}\Big)^{2n-1}$. Thus, ${\rm Spc}\,(\xi)=\{-1\}$. By Theorem A, as ${\rm Spc}\,(\xi)\subset\R$, we have that $\xi$ is globally $C^\infty$ linearizable
via the standard map. Besides, by Corollary A, $\mathscr{F}_\xi$ is topologically equivalent to the radial foliation. 
\end{exampleA4}

\begin{exampleB} Let $\varphi:\R^2\to\R^2$ be the orientation--reversing $C^\infty$ involution defined by
$$\varphi(x,y)=\left (
{\rm arcsinh}\,\Big (\dfrac{ {\rm sinh}\,(x)+{\rm sinh}\,(y)}{2} \Big),
{\rm arcsinh}\, \Big (\dfrac{3{\rm sinh}\,(x)-{\rm sinh}\,(y)}{2} \Big ) \right ).
$$

One may show that $\varphi$ is a $C^\infty$ involution such that:

\begin{equation*}
[D\varphi(p)]_E=
\begin{pmatrix}
\dfrac{{\rm cosh}\,(x)}{\sqrt{4+({\rm sinh}\,(x)+{\rm sinh}\,(y))^2 }} &  \dfrac{{\rm cosh}\,(y)}{\sqrt{4+({\rm sinh}\,(x)+{\rm sinh}\,(y))^2 }}\\
\dfrac{3{\rm cosh}\,(x)}{\sqrt{4+(3{\rm sinh}\,(x)-{\rm sinh}\,(y))^2 }} & -\dfrac{{\rm cosh}\,(y)}{\sqrt{4+(3{\rm sinh}\,(x)-{\rm sinh}\,(y))^2 }}
\end{pmatrix}.
\end{equation*}

Notice that for each $p\in\R^2$ there exist positive real numbers $a,b,c,d>0$, depending on $p$, such that the matrix of $D\varphi(p)$ with respect to the canonical basis is given by:

\begin{equation*}
[D\varphi(0)]_E=\frac12
\begin{pmatrix}
1 & 1\\
3 & -1
\end{pmatrix},\\ \quad
[D\varphi(p)]_E=
\begin{pmatrix}
a & b\\
c & -d
\end{pmatrix}.
\end{equation*}
Hence, 

\begin{equation*}
[D\varphi(0)D\varphi(p)]_E=\frac12
\begin{pmatrix}
a+c & b-d\\
3a-c & 3b+d
\end{pmatrix}.
\end{equation*}
In this way, ${\rm Trace}\,(D\varphi(0)D\varphi(p))=a+3b+c+d>0>-1$.
By Theorem B, $h$ is a global $C^\infty$ linearization of $\varphi$. By Corollary $B$,
$\mathscr{F}_\varphi$ is a $\varphi$-invariant $C^1$ foliation of $\R^2$ topologically equivalent to the vertical foliation.
\end{exampleB}

\begin{exampleC} In this example we deform the orientation--preserving linear involution $\varphi=-I$
to obtain a nonlinear orientation--preserving involution $\psi:\R^2\to\R^2$ that is not $C^1$ linearizable via the standard map. To proceed with the construction, let $\eta:\R\to [0,\pi]$ be a $C^1$ bump function such that $\eta^{-1}(\pi)=[-1,1]$ and
$\eta^{-1}(0)=\R\setminus (-2,2)$. Now let $\phi:\R^2\to [-\pi,\pi]$ be the $C^1$ function defined by $$\phi(p)=\eta\big(\Vert p-(3,3) \Vert^2\big)-\eta\big(\Vert p-(-3,-3) \Vert^2\big).$$
Notice that the function $\phi$ associates to each $p\in\R^2$ an angle $\phi(p)$. Points close to $(3,3)$ 
are associated the angle $\pi$ whereas points close to $(-3,-3)$ are associated the angle $-\pi$. Besides,
$\phi$ vanishes on $A=\R^2\setminus \big(B_2(3,3)\cup B_2(-3,-3)\big)$, where $B_r(p)$ denotes the closed ball of ratio $r$ centered at $p$.
Given $\theta\in\R$, let $R_\theta:\R^2\to\R^2$ be the rotation by $\theta$ defined by
$$
R_\theta
\begin{pmatrix}
x\\ y
\end{pmatrix}=
\begin{pmatrix}
{\rm cos}\, \theta & {\rm sin}\,\theta\\
-{\rm sin}\,\theta & {\rm cos}\,\theta
\end{pmatrix} 
\begin{pmatrix}
x\\y
\end{pmatrix}.
$$
Let $\rho_+:\R^2\to\R^2$ and $\rho_-:\R^2\to\R^2$ be defined by
\begin{eqnarray*}
\rho_+(p)&=&(3,3)+R_{\phi(p)}\big(p-(3,3)\big)\\
\rho_-(p)&=&(-3,-3)+R_{\phi(p)}\big(p-(-3,-3)\big),
\end{eqnarray*}
where $p$ varies on $\R^2$.
Let $\xi:\R^2\to[0,1]$ be a $C^1$ function such that $\xi\vert_{B_2(3,3)}=1$ and $\xi\vert_{B_2(-3,-3)}=0$.
We may choose $\xi$ to be invariant under rotations centered at $(3,3)$ or $(-3,-3)$, that is, we may suppose that $\xi\circ \rho_+=\xi$ and $\xi\circ\rho_-=\xi$.

Let $\rho:\R^2\to\R^2$ be the $C^1$ map defined by:
$$
\rho(p)=\xi(p)\rho_+(p)+(1-\xi(p))\rho_-(p),\quad p\in\R^2.
$$
The map $\rho$ takes each circle $C$ of ratio $r$ centered at $(3,3)$ $($respectively at $(-3,-3))$ onto itself acting as a rotation by $\phi(3+r,3)$ $($respectively by $\phi(-3+r,-3))$ around $(3,3)$ $($respectively $(-3,-3))$. Furthermore, $\phi$ is the identity map on $A$.

Let $\psi=\rho\circ \varphi$. We claim that $\psi$ is an orientation--preversing $C^1$ involution. Firstly observe that $\phi(-p)=-\phi(p)$ for all $p\in\R^2$. In other words, $\phi\circ\varphi=-\phi$. By the definiton of $\xi$ and $\rho$, we have that $\phi\circ \rho=\phi$. In this way,
for all $p\in\R^2$ we have that:
$$\psi(p)=\rho\circ\varphi(p)=R_{\phi\circ\varphi(p)}(-p)=R_{-\phi(p)}(-p).$$ 
It follows from the above that
\begin{eqnarray*}
\psi^2(p)&=&\rho\circ\varphi\circ\psi(p)=R_{\phi\circ\varphi\circ\psi(p)}\varphi\circ\psi(p)=
R_{-\phi\circ\psi(p)}\big(-R_{-\phi(p)}\big)(-p)=\\
&=&R_{-\phi\circ\rho\circ\varphi(p)}R_{-\phi(p)}(p)=R_{-\phi\circ\varphi(p)}R_{-\phi(p)}(p)=
R_{\phi(p)}R_{-\phi(p)}(p)=p.
\end{eqnarray*}
Now by the definition of $\rho$ and as $\varphi(B_1(3,3))=B_1(-3,-3)$, we have that for all $p$ in the interior of $B_{1}(3,3)$,
$$D\psi(p)=D(\rho\circ\varphi)(p)=D\rho(\varphi(p))D\varphi(p)=R_{-\pi}\varphi=I.$$
In this way, ${\rm Spc}\,(\psi)\supset \{1\}$ and thus Theorem A and Corollary A cannot be applied.
Notice that $g(p)=D\psi(0)(p)+\psi(p)=-p+(p-(6,6))=(-6,-6)$ for all $p\in B_{1}(3,3)$. In this way, $h=\frac12 D\psi(0)g$ is not injective and $\mathscr{F}_\psi$ is very degenerate.
\end{exampleC}

\begin{exampleD} Let $\gamma:\R\to\R^2$ be a $C^1$ embedding of the real line such that $\gamma\big((-\infty,1]\big)=(-\infty,1]\times\{0\}$ and $\gamma\cap \big(\{0\}\times (-\infty,+\infty) \big)=\{0,p,q\}$,
where $0=(0,0)$, $p=(0,a)$ and $q=(0,b)$. Let $\varphi:\R^2\to\R^2$ be an orientation--reversing $C^\infty$ involution such that
${\rm Fix}\,(\varphi)=\gamma(\R)$ and $D\varphi(0):(x,y)\mapsto (x,-y)$. Let $h:\R^2\to\R^2$ be the standard map for $\varphi$, that is, $h=I+D\varphi(0)\varphi$. We have that $\varphi(p)=p$, $\varphi(q)=q$ and
\begin{eqnarray*}
h(p)&=&h(0,a)=(0,a)+D\varphi(0)\varphi\big((0,a)\big)=(0,a)+(0,-a)=(0,0)\\
h(q)&=&h(0,b)=(0,b)+D\varphi(0)\varphi\big((0,b)\big)=(0,b)+(0,-b)=(0,0).
\end{eqnarray*}
Hence, $h$ is not injective and thus $\varphi$ is not $C^1$ linearizable via the standard map.
\end{exampleD}


\bibliographystyle{amsplain}

\begin{thebibliography}{99}

\bibitem{AGG} B. Alarc\'{o}n, V. Gu\' {i}\~nez and C. Gutierrez, Planar embeddings with a globally
attracting fixed point, {\it Nonlinear Anal.} {\bf 69} (2008), 140--150.

\bibitem{AGA} B. Alarc\'{o}n, C. Gutierrez and J. Mart\'{i}nez--Alfaro, Planar maps whose second iterate has a unique fixed point, {\it J. Difference Equ. Appl.} {\bf 4} (4) (2008), 421--428.

\bibitem{BS} F. Braun and J. R. dos Santos Filho, The Real Jacobian Conjecture on $\R^2$ is true
when one of the components has degree 3, {\it Discrete Cont. Dyn. Syst.} {\bf 26} (1) (2010), 75--87.

\bibitem{CM} M. Chamberland and G. Meisters, A mountain pass to the Jacobian Conjecture, {\it Canad. Math. Bull.} {\bf 41} (1998), 442--451.

\bibitem{C} M. Chamberland, Characterizing two--dimensional maps whose Jacobians have constant eigenvalues, {\it Canad. Math. Bull.} {\bf 46} (3) (2003),  323--331.

\bibitem{CGL} M. Cobo, C. Gutierrez and J. Llibre, On the injectivity of $C^1$ maps of the real plane, {\it Canad. J. Math.} {\bf 54} (6) (2002), 1187--1201.

\bibitem{FGR1} A. Fernandes, C. Gutierrez and R. Rabanal, On local diffeomorphisms of $\R^n$ that are injective, {\it Qual. Theory Dyn. Syst.} {\bf 4} (2004), 255--262.

\bibitem{FGR2} A. Fernandes, C. Gutierrez and R. Rabanal, Global asymptotic stability for differentiable vector fields of $\R^2$, {\it J. Differential Equations} {\bf 206} (2004) 470--482.

\bibitem{G1} C. Gutierrez, A solution to the bidimensional global asymptotic stability conjecture, {\it Ann. Inst. H. Poincar\'e Anal. Non Lin{\'e}aire} {\bf 12} (1995) 627--671. 

\bibitem{GJLT} C. Gutierrez, X. Jarque, J. Llibre and M. A. Teixeira,
Global injectivity of $C^r$ maps of the real plane, inseparable leaves and 
the Palais--Smale condition, {\it Canad. Math. Bull.} {\bf 50} (2007), 377--389.

\bibitem{GR} C. Gutierrez and R. Rabanal, Injectivity of differentiable maps $\R^2\to\R^2$ at infinity, {\it Bull. Braz. Math. Soc.} {\bf 37} (2), 217--239.

\bibitem{GPR} C. Gutierrez, B. Pires and R. Rabanal, Asymptotic stability at infinity for differentiable vector fields of the plane, {\it J. Differential Equations} {\bf 231} (2006).  

\bibitem{GS} C. Gutierrez and A. Sarmiento, Injectivity of $C^1$ maps $\R^2\to\R^2$ at infinity and planar
vector fields, {\it Ast{\'e}risque} {\bf 287} (2003), 89--102.

\bibitem{GT} C. Gutierrez and M. A. Teixeira, Asymptotic stability at infinity of planar vector fields, {\it Bull. Braz. Math. Soc.} {\bf 26} (1995), 57--66.

\bibitem{F} R. Fessler, A proof of the two dimensional Markus-Yamabe stability conjecture and a generalization, {\it Ann. Polon. Math.} {\bf LXII} (1995), 45--74.

\bibitem{H} M. Hirsch, Fixed--point indices, homoclinic contacts, and dynamics of 
injective planar maps, {\it Michigan Math. J.} {\bf 47} (2000), 101--108.


\bibitem{MMT} S. Mancini, M. Manoel and M. A. Teixeira, Divergent diagrams of folds and simultaneous
conjugacy of involutions, {\it Discrete Cont. Dyn. Syst.} {\bf 12} (4) (2005), 657--674.

\bibitem{Mac} R. S. MacKay, {\it Renormalisation in area-preserving maps}, World Scientific, 1993.

\bibitem{MY} L. Markus and H. Yamabe, Global stability criteria for differential systems, {\it Osaka J. Math} {\bf 12}
(1960), 305--317.

\bibitem{KMey} K. Meyer, Hamiltonian systems with a discrete symmetry, {\it J. Diff. Equations} {\bf 41} (1981), 228--238.

\bibitem{MZ} D. Montgomery and L. Zippin, {\it Topological transformation groups}, Interscience Publishers, 1955.

\bibitem{O} C. Olech, On the global stability of an autonomous system on the plane, {\it Contributions to Differential Equations} {\bf 1} (1993), 389--400.

\bibitem{R} R. Rabanal, Center type perfomance of differentiable vector fields in the plane, {\it Proc. Amer. Math. Soc.} {\bf 137} (2) (2009), 653--662.

\bibitem{SX} B. Smyth and F. Xavier, Injectivity of local diffeomorphisms from nearly spectral conditions,
{\it J. Differential Equations} {\bf 130} (1996), 406--414.

\bibitem{T} M. A. Teixeira, Local and simultaneous structural stability of certain diffeomorphisms, Lecture Notes in Math. {\bf 898}, Springer, 1981.

\end{thebibliography}

\end{document}